\documentclass[12pt]{amsart}
\usepackage{amsfonts}
\usepackage{amsmath}
\usepackage{amssymb}
\usepackage{mathrsfs}
\usepackage{a4}
\usepackage{amscd}

\newcommand{\heute}{14 August 2008}

\theoremstyle{plain}
\newtheorem{theorem}{Theorem}[section]

\newtheorem{lemma}[theorem]{Lemma}
\newtheorem{corollary}[theorem]{Corollary}

\newtheorem{conjecture}[theorem]{Conjecture}
\theoremstyle{remark}
\newtheorem{remark}[theorem]{Remark}
\newtheorem*{defn}{Definition}
\newtheorem*{rk}{Remark}

\newcommand{\dashTwo}[1]{\textup{(\ref{two}${}'$)}}

\newcommand{\ignore}[1]{}

\newcommand{\f}[1][p]{\mathbb{F}_{#1}}

\newcommand{\Gro}[1]{Gr\"ob\-ner}

\DeclareMathOperator{\rank}{rank}

\newcommand{\GL}{\mathit{GL}}
\DeclareMathSymbol\normal{\mathrel}{AMSa}{"43}

\newcommand{\abs}[1]{\left|#1\right|}

\begin{document}

\title{On Oliver's $p$-group conjecture}
\author[D.~J. Green]{David J. Green}
\address{Dept of Mathematics \\
Friedrich-Schiller-Universit\"at Jena \\ 07737 Jena \\ Germany}
\email{green@minet.uni-jena.de}
\author[L.~H\'ethelyi]{L\'aszl\'o H\'ethelyi}
\address{Department of Algebra \\
Budapest University of Technology and Economics \\ Budapest \\ Hungary}
\author[M.~Lilienthal]{Markus Lilienthal}
\address{Dept of Mathematics \\
Friedrich-Schiller-Universit\"at Jena \\ 07737 Jena \\ Germany}
\curraddr{%
Professur f\"ur Electronic Commerce \\
FB Wirtschaftswissenschaften \\
Johann Wolfgang Goethe-Universit\"at \\
Postfach 84 \\
Mertonstr.\@ 17-25 \\
D-60054 Frankfurt \\ Germany}
\date{\heute}

\begin{abstract}
\noindent
Let $S$ be a $p$-group for an odd prime~$p$.
B.~Oliver conjectures that a certain characteristic subgroup $\mathfrak{X}(S)$
always contains the Thompson subgroup~$J(S)$.
We obtain a reformulation of the conjecture as a statement
about modular representations of $p$-groups.
Using this we verify Oliver's conjecture for groups where
$S/\mathfrak{X}(S)$ has nilpotence class at most two.
\end{abstract}

\maketitle

\section{Introduction}
\label{sect:intro}
\noindent
The recently introduced concept of a $p$-local finite group seeks
to provide a treatment of the $p$-local structure of a finite
group~$G$ which does not refer directly to the group~$G$ itself and yet retains
enough information to construct the $p$-localisation of the classifying
space~$\mathit{BG}$\@. Ideally one could then associate a $p$-local
classifying space to a $p$-block of~$G$, and to certain exotic fusion systems.
See the survey article~\cite{BLO:survey} by Broto, Levi and Oliver for an
introduction to this area.

A key open question about $p$-local finite groups is whether or not there is
a unique centric linking system associated to each saturated fusion system.
Oliver showed that this would follow from a conjecture about higher limits
(Conjecture 2.2 in~\cite{Oliver:MartinoPriddyOdd}); and that for
odd primes this higher
limits conjecture would in turn
follow from the following purely group-theoretic conjecture:
\medskip

\paragraph{\textbf{Oliver's Conjecture 3.9}} (\cite{Oliver:MartinoPriddyOdd})
\quad
Let $S$ be a $p$-group for an odd prime~$p$. Then
\[
J(S) \leq \mathfrak{X}(S) \, ,
\]
where $J(S)$ is the Thompson subgroup generated
by all elementary abelian $p$-subgroups whose rank is the $p$-rank of~$S$,
and $\mathfrak{X}(S)$ is the Oliver subgroup described
in~\S\ref{sect:manageable}\@.
\medskip

\noindent
Our main result on Oliver's conjecture is as follows:

\begin{theorem}
\label{thm:partial}
Let $S$ be a $p$-group for an odd prime~$p$. If $S/\mathfrak{X}(S)$ has
nilpotency class at most two, then $S$ satisfies Oliver's conjecture.
\end{theorem}

\begin{rk}
This subsumes all three cases of Oliver's Proposition~3.7 in the first case
$\mathfrak{X}(S) \geq J(S)$.
\end{rk}

\noindent
The proof of Theorem~\ref{thm:partial} depends on a reformulation of
Oliver's conjecture, for which we need to recall the terms $F$-module and
offender. See e.g.~\cite{MeierfStellm:OtherPGV} for a recent paper about
offenders.

\begin{defn}[Definition~26.5 in~\cite{RevisionVol2}]
Let $G$ be a finite group and $V$ a faithful $\f G$-module. If there exists
a non-identity elementary abelian $p$-subgroup
$E \leq G$ which satisfies the inequality $\abs{E} \abs{C_V(E)} \geq \abs{V}$,
then $V$ is called an \emph{$F$-module} for~$G$, and $E$ an
\emph{offending subgroup}.
\end{defn}

\begin{rk}
$F$-module is short for ``failure of (Thompson) factorization module''.
Another way to phrase the inequality is $\dim(V) - \dim(V^E) \leq \rank(E)$.
\end{rk}

\noindent
We will always take $G$ to be a nontrivial $p$-group.  Hence the
$\f G$-module $V$ is
faithful if and only if it is faithful as a module for $\Omega_1(Z(G))$. 
We shall be interested in the following stronger condition:
\begin{description}
\item[(PS)]
The restriction of $V$ to each central order~$p$ subgroup has a
nontrivial projective summand.
\end{description}

\begin{rk}
Projective and free are equivalent here. We are grateful to the referee for
suggesting this formulation of the property. Another formulation is
that every central order~$p$ element operates with minimal polynomial
$(X-1)^p$: equivalence follows from the standard properties of the Jordan
normal form.
\end{rk}


\begin{theorem}
\label{thm:main}
Let $G\neq 1$ be a finite $p$-group.
Then Oliver's conjecture holds for every
finite $p$-group $S$ with $S/\mathfrak{X}(S) \cong G$ if and only if
$G$~has no $F$-modules satisfying~(PS)\@.
\end{theorem}

\begin{conjecture}
\label{conj:GHL}
Let $p$ be an odd prime and $G \neq 1$ a finite $p$-group. Then
$G$ has no $F$-modules which satisfy~(PS)\@.
\end{conjecture}

\begin{corollary}
\label{coroll:main}
Conjecture~\textnormal{\ref{conj:GHL}} is equivalent to Oliver's
Conjecture~\textnormal{3.9}\@.
\end{corollary}

\noindent
We prove Theorem~\ref{thm:partial} by verifying Conjecture~\ref{conj:GHL}
for groups of class at most two. For this we need the following result.

\begin{defn}[See~\cite{Glauberman:Quadratic}]
Let $V$ be a faithful $\f G$-module. A non-identity element $g \in G$ is
called \emph{quadratic} if $(g-1)^2V = 0$.
\end{defn}

\begin{theorem}
\label{thm:Class2}
Suppose that $p$~is an odd prime, $G$ is a $p$-group of nilpotence class at most
two, and $V$ is a faithful $\f G$-module. If $G$ contains a quadratic element,
then so does $\Omega_1(Z(G))$.
\end{theorem}

\paragraph{\emph{Structure of the paper}}
We prove Theorem~\ref{thm:main} and Corollary~\ref{coroll:main} in
\S\ref{sect:manageable}\@.
In \S\ref{sect:Replacement} we derive a consequence of the Replacement Theorem,
Theorem~\ref{thm:serious}\@.
Then in~\S\ref{sect:partial} we prove Theorems \ref{thm:Class2}~and
\ref{thm:partial}.
Finally in \S\ref{section:class3}
we discuss a class three example which cannot be handled using
Theorem~\ref{thm:serious}\@.


\section{The reformulation of Oliver's conjecture}
\label{sect:manageable}
\noindent
For the convenience of the reader we start by recapping the definition
and elementary properties of~$\mathfrak{X}(S)$, as given in~\S3
of Oliver's paper~\cite{Oliver:MartinoPriddyOdd}\@.

\begin{defn}[c.f.~\cite{Oliver:MartinoPriddyOdd}, Def.~3.1]
Let $S$ be a $p$-group and $K \normal S$ a normal subgroup. A $Q$-series
leading up to~$K$ consists of a series of subgroups
\[
1 = Q_0 \leq Q_1 \leq \cdots \leq Q_n = K
\]
such that each $Q_i$ is normal in~$S$, and such that
\begin{equation*}
[\Omega_1(C_S(Q_{i-1})),Q_i;p-1] = 1
\end{equation*}
holds for each $1 \leq i \leq n$.
The unique largest normal subgroup of~$S$ which admits such a $Q$-series
is called~$\mathfrak{X}(S)$, the Oliver subgroup of~$S$.
\end{defn}

\begin{lemma}[Oliver]
\label{lemma:oliver}
If $1 = Q_0 \leq Q_1 \leq \cdots \leq Q_n = K$ is such a $Q$ series
and $H \normal G$ also admits a $Q$-series, then there is a $Q$-series
leading up to $HK$ which starts with $Q_0,\ldots,Q_n$\@.

Hence there is indeed a unique largest subgroup admitting a $Q$-series,
and this subgroup $\mathfrak{X}(S)$ is characteristic in~$S$.
In addition, $\mathfrak{X}(S)$ is centric in~$S$: recall that $P \leq S$ is
centric if $C_S(P)=Z(P)$.
\end{lemma}

\begin{proof}
See pages 334--5 of Oliver's paper~\cite{Oliver:MartinoPriddyOdd}\@.
\end{proof}

\noindent
Now we can start to derive the reformulation of Oliver's conjecture.

\begin{lemma}
\label{lemma:dxxplyRooted}
Let $S$ be a finite $p$-group with $\mathfrak{X}(S) < S$. Then the induced
action of $G := S/\mathfrak{X}(S)$ on $V := \Omega_1(Z(\mathfrak{X}(S)))$
satisfies~(PS)\@.
\end{lemma}

\begin{proof}
Pick $g \in S$ such that $1 \neq g\mathfrak{X}(S) \in \Omega_1(Z(G))$. Then
$\langle \mathfrak{X}(S),g\rangle \normal S$ and so $[V,g;p-1] \neq 1$,
by maximality of~$\mathfrak{X}(S)$. So the minimal polynomial
of the action of~$g$ does not divide $(X-1)^{p-1}$. But it has to divide
$(X-1)^p = X^p - 1$. So $(X-1)^p$ is the minimal polynomial.
This is the reformulation of~(PS)\@.
\end{proof}

\begin{proof}[Proof of Theorem~\ref{thm:main}]
Suppose first that no $F$-module for $G$ satisfies $(PS)$, and
that $S/\mathfrak{X}(S) \cong G$. Let us prove Oliver's Conjecture for~$G$.
By Lemma~\ref{lemma:dxxplyRooted} the induced action of $G$ on
$V := \Omega_1 (Z(\mathfrak{X}(S)))$ satisfies~(PS), so by assumption there
are no offending subgroups.

Let $E \leq S$ be an elementary abelian subgroup not contained
in~$\mathfrak{X}(S)$. It suffices for us to show that $\mathfrak{X}(S)$
contains an elementary abelian of greater rank than~$E$.
We can split $E$ up as $E = E_1 \times E_2 \times E_3$, with
$E_1 = E \cap V \leq V^E$ and
$E_1 \times E_2 = E \cap \mathfrak{X}(S)$. By assumption, $1 \neq E_3$
embeds in $S/\mathfrak{X}(S) \cong G$. As there are no offenders,
we have $\dim(V) - \dim(V^{E_3}) > \rank(E_3)$.
But $V^{E_3} = V^E$.
So $V \times E_2$ lies in $\mathfrak{X}(S)$ and has greater rank
than~$E$.

Conversely suppose that the $\f G$-module $V$ is an $F$-module and
satisfies~(PS)\@.
Set $S$ to be the semidirect product $S = V \rtimes G$
defined by this action. From Lemma~\ref{lemma:semidirect} below we see
that $V = \mathfrak{X}(S)$.
As $V$~is an $F$-module, there is an offender:
an elementary abelian subgroup
$1 \neq E \leq G$ with $\dim(V) - \dim(V^E) \leq \rank(E)$. This means
that $W := V^E \times E$ is an elementary abelian subgroup which does not
lie in $V = \mathfrak{X}(S)$ but does have rank at least as great as that
of~$\mathfrak{X}(S)$. So $W \leq J(S)$ and
therefore $J(S) \nleq \mathfrak{X}(S)$.
\end{proof}

\begin{lemma}
\label{lemma:semidirect}
Suppose that $V$ is an $\f G$-module which satisfies~(PS)\@.
Let $S$ be the semidirect product $S = V \rtimes G$
defined by this action. Then $V = \mathfrak{X}(S)$.
\end{lemma}

\begin{proof}
First we prove that $V$ is a maximal normal abelian subgroup
of~$S$: clearly it is abelian and normal. If $A$
is a normal abelian subgroup strictly containing $V$,
then $A = V \rtimes H$ for some nontrivial abelian $H \normal G$. As
$H$ is nontrivial and normal it contains an order~$p$ element $g$~of $Z(G)$.
Since $V$~satisfies (PS), it follows that $g$~acts on~$V$ with minimal
polynomial $(X-1)^p$. But that is a contradiction, as $A$ is abelian. So
$V$ is indeed maximal normal abelian.

We now argue as in the proof of Oliver's Lemma~3.2\@. Since $V$ is maximal
normal abelian, it is centric in~$S$: for if not then
$V < C_S(V) \normal S$, and so $C_S(V)/V$ has nontrivial
intersection with the centre of $S/V$. Picking an $x \in C_S(V)$ whose image
in $C_S(V)/V$ is a nontrivial element of this intersection, we obtain a
strictly larger normal abelian subgroup $\langle V, x\rangle$,
a contradiction. Hence $\Omega_1 C_S(V)=V$.

Moreover, since $V$ is normal abelian and $p > 2$, there is a $Q$-series
$1 < V$. So by Lemma~\ref{lemma:oliver} there is a $Q$-series leading
up to~$\mathfrak{X}(S)$ with $Q_1 = V$.
If $V < \mathfrak{X}(S)$ then there is $Q_1 < Q_2 \normal S$
with $[V,Q_2;p-1]=1$. But this cannot happen, because by the argument
of the first paragraph of this proof there is a $g \in Q_2$ whose action
on~$V$ has minimal polynomial $(X-1)^p$. So $V = \mathfrak{X}(S)$.
\end{proof}

\begin{proof}[Proof of Corollary~\ref{coroll:main}]
Immediate from Theorem~\ref{thm:main}\@.
If $\mathfrak{X}(S) = S$ then Oliver's Conjecture holds automatically.
\end{proof}

\section{The Replacement Theorem}
\label{sect:Replacement}
\noindent
We shall need the following lemma,
which is a special case of the
Replacement Theorem and its proof in~\cite[X, 3.3]{BlackburnHuppert:III}\@.

\begin{lemma}
\label{lemma:Replacement}
Suppose that $G \neq 1$ is elementary abelian, that $V$ is a faithful
$\f G$-module, and that $G$ contains no quadratic elements. Let us write
\[
T = \{(H,W) \mid \text{$H \leq G$ and $W$ is a subspace of $V^H$}\} \, .
\]
Suppose that $(H,W) \in T$ with $H \neq 1$. Then there is $(K,U) \in T$
with $K < H$, $W \subsetneq U \subsetneq V$ and
$\left|H \times W\right| = \left|K \times U\right|$.
\end{lemma}

\begin{proof}
Let us set
$I = \{v \in V \mid \text{$(h-1)v \in W$ for every $h \in H$}\}$ and
$J = \{v \in V \mid \text{$(h-1)v \in I$ for every $h \in H$}\}$.
If $1 \neq h \in H$ then $(h-1)^2 v \neq 0$ for some $v \in V$.
Then $v \not \in I$, for otherwise $(h-1)v \in W$ and so $(h-1)^2 v = 0$.
So $I \subsetneq V$, and therefore $W \subsetneq I \subsetneq J$ by the usual
orbit length argument. Pick
$v_0 \in J \setminus I$ and set $U$~to be the subspace spanned
by $W$~and $\{(h-1)v_0 \mid h \in H\}$. Set $K = \{h \in H \mid
(h-1)v_0 \in W\}$.
So $U \supsetneq W$ by choice of~$v_0$. Also $U \subseteq I \subsetneq V$.
If $h,h' \in H$ then
$(hh'-1)v_0 = (h-1)v_0 + (h'-1)v_0 + (h-1)(h'-1)v_0$, and so
\begin{equation}
\label{eqn:Replacement}
(hh'-1)v_0 \equiv (h-1)v_0 + (h'-1)v_0 \pmod{W} \, .
\end{equation}
So $K \leq H$, and in fact $K < H$ by choice of~$v_0$.
By Eqn.~\eqref{eqn:Replacement} it also follows that $\left|H:K\right| = p^r$
for $r = \dim U - \dim W$.
Finally $U \subseteq V^K$, for if $k \in K$ and $u \in U$,
then
\[
u = \sum_{h \in H} \lambda_h (h-1)v_0 + w
\]
for suitable $\lambda_h \in \f$, $w \in W$. So
\[
(k-1)u = \sum_{h \in H} \lambda_h (h-1)(k-1) v_0 = 0 \, ,
\]
since $(k-1)v_0 \in W \subseteq V^H$.
\end{proof}

\begin{corollary}
\label{coroll:Replacement}
Suppose as in Lemma~\ref{lemma:Replacement} that $(H,W) \in T$
and $H \neq 1$. Then $\left|H \times W\right| < \left|V\right|$.
\end{corollary}

\begin{proof}
By induction on~$\left|H\right|$. By the lemma we may reduce $\left|H\right|$
whilst keeping $\left|H \times W\right|$ constant. This process only stops
when we arrive at $(K,U)$ with $K=1$. But $U \subsetneq V$ by the lemma.
\end{proof}

\noindent
The following result is presumably well known to those familiar with
Thompson factorization.

\begin{theorem}
\label{thm:serious}
Suppose that $p$~is an odd prime, $G$ is a finite group,
$V$ is a faithful $\f G$-module, and $E \leq G$ is a non-identity
elementary abelian $p$-subgroup.
If $E$ is an offender, then it must contain a quadratic element.
\end{theorem}

\begin{proof}
Without loss of generality $E = G$.
Apply Corollary~\ref{coroll:Replacement} to the pair $(G,V^G) \in T$.
\end{proof}

\begin{rk}
Pursuing this direction further, it might be worthwhile to investigate
potential applications of the $P(G,V)$-theorem in the theory of $p$-local
finite groups.  The properties of the Thompson subgroup $J(S)$ which
Chermak describes in his comments on the motivation for
the $P(G,V)$-theorem \cite[Rk 2]{Chermak:QuadraticAction}
are the same properties which led to $J(S)$ featuring in Oliver's conjecture.
And Timmesfeld's replacement theorem plays an important part in the proof of
the $P(G,V)$-theorem.
\end{rk}

\section{Nilpotence class at most two}
\label{sect:partial}
\noindent
We can now start work on the proof of Theorem~\ref{thm:partial}\@.

\begin{lemma}
\label{lemma:Class2}
Suppose that $p$ is an odd prime, that $G\neq1$ is a finite $p$-group, and
that $V$~is a faithful $\f G$-module.
Suppose that $A,B \in G$ are such that $C := [A,B]$ is a nontrivial element
of $C_G(A,B)$. If $C$ is non-quadratic, then so are $A$~and $B$.
\end{lemma}

\begin{proof}
By symmetry it suffices to prove that $B$ is non-quadratic. So suppose
that $B$~is quadratic.
Denote by $\alpha,\beta,\gamma$ the action matrices on~$V$ of
$A-1$, $B-1$ and $C-1$ respectively.

By assumption we have
$\gamma^2 \neq 0$ and $\beta^2 = 0$.
As $C$~commutes with $A$~and $B$, we have
$\alpha \gamma = \gamma \alpha$ and $\beta \gamma = \gamma \beta$.
Since $[A,B]=C$, we have $AB = BAC$ and therefore
\begin{equation}
\label{eqn:Class2}
\alpha \beta - \beta \alpha = \gamma (1 + \beta + \alpha + \beta \alpha) \, .
\end{equation}
Evaluating $\beta \cdot \text{Eqn.~\eqref{eqn:Class2}} \cdot \beta$,
we deduce that
$\gamma \beta \alpha \beta = 0$.
So when we evaluate $\beta \cdot \text{Eqn.~\eqref{eqn:Class2}}
+ \text{Eqn.~\eqref{eqn:Class2}} \cdot \beta$, we find that
$\gamma(2\beta + \beta \alpha + \alpha \beta) = 0$. Let us
write $\lambda = -\frac12$ and $\delta = \gamma \beta$. Then we have
\[
\delta = \lambda(\delta \alpha + \alpha \delta) \, .
\]
From this one sees by induction upon $r \geq 1$ that
\[
\delta = \lambda^r \sum_{s = 0}^r \binom{r}{s} \alpha^s \delta \alpha^{r-s} \, .
\]
As $A$ has order a power of~$p$, it follows that $(A-1)$ and its action
matrix~$\alpha$ are nilpotent. From this we deduce that $\delta = 0$,
that is $\gamma \beta = 0$.
Applying this to $\gamma \cdot \text{Eqn.~\eqref{eqn:Class2}}$ we see that
$\gamma^2 (1 + \alpha) = 0$. As $\alpha$~is nilpotent it follows
that $\gamma^2 = 0$, a contradiction. So $\beta^2 \neq 0$ after all.
\end{proof}

\begin{proof}[Proof of Theorem~\ref{thm:Class2}]
We suppose that $\Omega_1(Z(G))$ has no quadratic elements, and show
that $G$ has none either.
Suppose $1 \neq B \in Z(G)$. Then is an $r \geq 0$ with
$1 \neq B^{p^r} \in \Omega_1(Z(G))$. So $B^{p^r}$ is not quadratic.
Hence
$(B-1)^{2p^r} = (B^{p^r}-1)^2$ has nonzero action. So $(B-1)^2$
has nonzero action, and $Z(G)$ contains no quadratic elements.

If $B \not \in Z(G)$ then the nilpotency class is two and there is an element
$A \in G$ with $1 \neq [A,B] \in Z(G)$. So $(B-1)^2$ has nonzero action by
Lemma~\ref{lemma:Class2}\@.
\end{proof}

\begin{corollary}
\label{coroll:manClass2}
Suppose that $p$ is an odd prime, $G\neq1$ a finite $p$-group and
$V$ an $\f G$-module which satisfies~(PS)\@. If the nilpotence class of~$G$
is at most two then $V$~cannot be an $F$-module.
\end{corollary}

\begin{proof}
As $p$ is odd, condition (PS) means that there are no quadratic
elements in $\Omega_1(Z(G))$. Then Theorem~\ref{thm:Class2} says
that there are no quadratic elements in~$G$.
So by Theorem~\ref{thm:serious} there are no offenders.
\end{proof}

\begin{proof}[Proof of Theorem~\ref{thm:partial}]
Follows from 
Corollary~\ref{coroll:manClass2}
and
Theorem~\ref{thm:main}
if $\mathfrak{X}(S) < S$.
If $\mathfrak{X}(S)=S$ then there is nothing to prove.
\end{proof}

\section{A class 3 example}
\label{section:class3}
\noindent
Theorem~\ref{thm:Class2} was a key step in the proof of
Theorem~\ref{thm:partial}\@. We now give an example which shows that
Theorem~\ref{thm:Class2} does not apply to groups of nilpotence class
three.

Let $G$ be the semidirect product $G = K \rtimes L$, where
the $K = \f[3]^3$ is elementary abelian of order~$3^3$,
$L = \langle A \rangle$ is cyclic of order~$3$, and the action of $L$~on
$v \in K$ is given by
\[
A v A^{-1} = \begin{pmatrix} 1 & 1 & 0 \\ 0 & 1 & 1 \\ 0 & 0 & 1 \end{pmatrix}
\cdot v
\, .
\]
Observe that $G$~is isomorphic to the wreath product $C_3 \wr C_3$, as the
action of~$A$ permutes the following basis of~$K$ cyclically:
$(0,0,1)$, $(0,1,1)$, $(1,2,1)$.

Setting $B=(0,0,1)$, $C=(0,1,0)$ and $D=(1,0,0)$ we obtain the following
presentation of~$G$, where we take $[A,B]$ to mean $ABA^{-1}B^{-1}$.
\[
G = \biggl\langle A,B,C,D \biggm| \parbox{200pt}{$A^3=B^3=C^3=D^3=1$, \quad
$D$ central, \\
$[B,C]=1$, \quad $[A,B]=C$, \quad $[A,C]=D$} \biggr\rangle \, ,
\]
From this we deduce that matrices
$\alpha,\beta,\gamma,\delta \in M_n(\f[3])$ induce a representation
$\rho \colon G \rightarrow \GL_n(\f[3])$ with
\begin{xalignat*}{4}
\rho(A) & = 1 + \alpha &
\rho(B) & = 1 + \beta &
\rho(C) & = 1 + \gamma &
\rho(D) & = 1 + \delta
\end{xalignat*}
if and only if the following relations are satisfied, where $[\alpha, \beta]$
now of course means $\alpha \beta - \beta \alpha$:
\begin{equation}
\label{eqn:alpharep}
\begin{aligned}
\alpha^3 = \beta^3 & = \gamma^3 = \delta^3 = 0 \\
[\alpha,\delta] = [\beta,\delta] & = [\gamma,\delta] = [\beta,\gamma] = 0 \\
[\alpha, \beta] = \gamma(1+\beta)(1+\alpha)
& \hspace*{3em}
[\alpha, \gamma] = \delta(1+\gamma)(1+\alpha)
\end{aligned}
\end{equation}
Now we consider what it means for such a representation to satisfy~(PS)\@.
Here, $Z(G)=\langle D\rangle$ is cyclic of order~$3$. So
we need both $(\rho(D)-1)^2$ and $(\rho(D^2) - 1)^2$ to be non-zero.
That is, $\delta^2$ and $(\delta^2 + 2\delta)^2 = \delta^2(1+\delta+\delta^2)$
should both be nonzero. But $1+\delta+\delta^2$ is invertible, since
$\delta$~is nilpotent.

We deduce therefore that matrices $\alpha,\beta,\gamma,\delta \in \GL_n(\f[3])$
induce a representation of~$G$ satisfying~(PS) if and only if they satisfy
the inequality
\begin{equation}
\label{eqn:deltadeep}
\delta^2 \neq 0
\end{equation}
in addition to the equations~\eqref{eqn:alpharep}\@.

Using GAP~\cite{GAP4} we obtained the
the following matrices in $\GL_8(\f[3])$. The reader is invited to
check\footnote{See
\texttt{http://www.minet.uni-jena.de/\~{}green/Documents/matTest.g}
for a GAP script which performs these checks.}
that they satisfy the
relations \eqref{eqn:alpharep}~and \eqref{eqn:deltadeep}\@.
\begin{xalignat*}{2}
\delta & =  \begin{pmatrix}
0 & 0 & 1 & 0 & 0 & 0 & 0 & 0 \\
0 & 0 & 0 & 1 & 0 & 0 & 0 & 0 \\
0 & 0 & 0 & 0 & 0 & 0 & 0 & 2 \\
0 & 0 & 0 & 0 & 0 & 0 & 0 & 1 \\
0 & 0 & 0 & 0 & 0 & 0 & 1 & 0 \\
0 & 0 & 0 & 0 & 0 & 0 & 0 & 1 \\
0 & 0 & 0 & 0 & 0 & 0 & 0 & 0 \\
0 & 0 & 0 & 0 & 0 & 0 & 0 & 0
\end{pmatrix}
&
\gamma & = \begin{pmatrix}
0 & 0 & 1 & 1 & 2 & 2 & 2 & 1 \\
0 & 0 & 1 & 0 & 1 & 1 & 0 & 0 \\
0 & 0 & 0 & 0 & 0 & 0 & 2 & 2 \\
0 & 0 & 0 & 0 & 0 & 0 & 1 & 0 \\
0 & 0 & 0 & 0 & 0 & 0 & 1 & 1 \\
0 & 0 & 0 & 0 & 0 & 0 & 1 & 0 \\
0 & 0 & 0 & 0 & 0 & 0 & 0 & 0 \\
0 & 0 & 0 & 0 & 0 & 0 & 0 & 0
\end{pmatrix}
\\
\beta & = \begin{pmatrix}
0 & 0 & 0 & 0 & 1 & 0 & 0 & 0 \\
0 & 0 & 0 & 0 & 0 & 1 & 0 & 0 \\
0 & 0 & 0 & 0 & 0 & 0 & 1 & 0 \\
0 & 0 & 0 & 0 & 0 & 0 & 0 & 1 \\
0 & 0 & 0 & 0 & 0 & 0 & 0 & 0 \\
0 & 0 & 0 & 0 & 0 & 0 & 0 & 0 \\
0 & 0 & 0 & 0 & 0 & 0 & 0 & 0 \\
0 & 0 & 0 & 0 & 0 & 0 & 0 & 0
\end{pmatrix}
&
\alpha & = \begin{pmatrix}
2 & 2 & 0 & 2 & 0 & 1 & 0 & 1 \\
1 & 1 & 2 & 2 & 0 & 0 & 0 & 0 \\
0 & 0 & 2 & 2 & 0 & 0 & 0 & 0 \\
0 & 0 & 1 & 1 & 0 & 0 & 0 & 0 \\
0 & 0 & 1 & 2 & 0 & 0 & 1 & 1 \\
0 & 0 & 1 & 1 & 0 & 0 & 0 & 0 \\
0 & 0 & 0 & 0 & 0 & 0 & 0 & 1 \\
0 & 0 & 0 & 0 & 0 & 0 & 0 & 0
\end{pmatrix}
\end{xalignat*}
Observe that $\beta^2=0$. So although this module satisfies~(PS),
the elementary abelian subgroups $\langle B \rangle$ and
$\langle B,C,D \rangle$ both contain $B$, a quadratic element.
So we must find another way to show that they are not offenders:
Theorem~\ref{thm:serious} does not apply.

\begin{remark}
More generally, we are not currently able to decide Conjecture~\ref{conj:GHL}
either way for the wreath product group $H \wr C_3$, where the group
$H$~on the bottom is an elementary abelian $3$-group.
\end{remark}

\subsection*{Acknowledgements}
We are grateful to George Glauberman for generously sharing his background
knowledge with us, and to the referee for advice concerning
terminology.

\bibliographystyle{abbrv}
\bibliography{../united}

\newcommand{\SortNoop}[1]{} \newcommand{\cocoa}{{\hbox{\rm C\kern-.13em
  o\kern-.07em C\kern-.13em o\kern-.15em A}}}
\begin{thebibliography}{1}

\bibitem{BLO:survey}
C.~Broto, R.~Levi, and B.~Oliver.
\newblock The theory of {$p$}-local groups: a survey.
\newblock In P.~Goerss and S.~Priddy, editors, {\em Homotopy theory: relations
  with algebraic geometry, group cohomology, and algebraic $K$-theory}, volume
  346 of {\em Contemp. Math.}, pages 51--84. Amer. Math. Soc., Providence, RI,
  2004.

\bibitem{Chermak:QuadraticAction}
A.~Chermak.
\newblock Quadratic action and the ${P}({G},{V})$-theorem in arbitrary
  characteristic.
\newblock {\em J. Group Theory}, 2:1--13, 1999.

\bibitem{GAP4}
The GAP~Group.
\newblock {\em {GAP -- Groups, Algorithms, and Programming, Version 4.4.10}},
  2007.
\newblock \verb+(http://www.gap-system.org)+.

\bibitem{Glauberman:Quadratic}
G.~Glauberman.
\newblock Quadratic elements in unipotent linear groups.
\newblock {\em J. Algebra}, 20:637--654, 1972.

\bibitem{RevisionVol2}
D.~Gorenstein, R.~Lyons, and R.~Solomon.
\newblock {\em The classification of the finite simple groups. {N}umber 2},
  volume~40 of {\em Mathematical Surveys and Monographs}.
\newblock American Mathematical Society, Providence, RI, 1996.

\bibitem{BlackburnHuppert:III}
B.~Huppert and N.~Blackburn.
\newblock {\em Finite groups. {III}}, volume 243 of {\em Grundlehren der
  Mathematischen Wissenschaften [Fundamental Principles of Mathematical
  Sciences]}.
\newblock Springer-Verlag, Berlin, 1982.

\bibitem{MeierfStellm:OtherPGV}
U.~Meierfrankenfeld and B.~Stellmacher.
\newblock The other ${P}({G},{V})$-theorem.
\newblock {\em Rend. Sem. Mat. Univ. Padova}, 115:41--50, 2006.

\bibitem{Oliver:MartinoPriddyOdd}
B.~Oliver.
\newblock Equivalences of classifying spaces completed at odd primes.
\newblock {\em Math. Proc. Cambridge Philos. Soc.}, 137(2):321--347, 2004.

\end{thebibliography}



\end{document}